\documentclass[11pt]{article}
\usepackage{amsmath,amssymb, amsfonts, amsthm, latexsym, hyperref,titlesec }
\usepackage{bbm,xcolor,graphicx}

\titleformat{\subsubsection}[runin] {\normalfont\bfseries}{\thesubsubsection}{0.7em}{\addperiod}
\newcommand{\addperiod}[1]{#1.}

\titlespacing\subsubsection{5pt}{4pt plus 4pt minus 2pt}{5pt plus 2pt minus 2pt}

\newtheorem{thm}{Theorem}

\newtheorem{lem}{Lemma}[section]
\newtheorem{lemma}[lem]{Lemma}

\newtheorem{propos}[lem]{Proposition}

\newtheorem{conjecture}{Conjecture}

\newcommand{\N}{\mathbb{N}}

\newcommand{\E}{\mathbb{E}}
\newcommand{\eps}{\varepsilon}
\newcommand{\Prob}{\mathbb{P}}
\renewcommand{\Pr}{\mathbb{P}}

\DeclareMathOperator{\Maxload}{MaxLoad}

\title{\vspace{-2.2cm}The power of thinning in balanced allocation}
\author{Ohad N. Feldheim\thanks{Hebrew University of Jerusalem, email:\texttt{Ohad.Feldheim@mail.huji.ac.il}.} \and Ori Gurel-Gurevich\thanks{Hebrew University of Jerusalem, email:\texttt{Ori.Gurel-Gurevich@mail.huji.ac.il}, research was supported by the Israel Science Foundation (grant No. 1707/16).}}

\begin{document}

\maketitle
\begin{abstract}
Balls are sequentially allocated into $n$ bins as follows: for each ball, an independent, uniformly random bin is generated. An overseer may then choose to either allocate the ball to this bin, or else the ball is allocated to a new independent uniformly random bin. The goal of the overseer is to reduce the load of the most heavily loaded bin after $\Theta(n)$ balls have been allocated.
We provide an asymptotically optimal strategy yielding a maximum load of $(1+o(1))\sqrt{\frac{8\log n}{\log\log n}}$ balls.
\end{abstract}
\smallskip
\noindent \textbf{Keywords.} Thinning, two-choices, one-retry, $(1+\alpha)$-choice, load  balancing, balls and bins, balanced allocation, subsampling.

\section{Introduction and Results}

Fix $\rho>0$ and consider a model in which an overseer is monitoring the sequential allocation of $\lfloor \rho n\rfloor$ balls into $n$ bins. Each ball is assigned a \emph{primary allocation}, i.e., an independent, uniformly chosen random bin. Then, the overseer is given the choice to reject this primary allocation, in which case the ball is assigned a \emph{secondary allocation} instead, that is, a new, independent, uniformly chosen random bin. The set of all resulting allocations is called a \emph{two-thinning} of the balls-and-bins process.

A \emph{two-thinning strategy}, is a function determining whether to accept or reject each suggested allocation, depending on all previous allocations. Denote by $\Maxload^f_{t}([n])$ the load of the most heavily loaded bin after the player allocates $\lfloor t\rfloor$ balls into $n$ bins, following the strategy $f$. A strategy is \emph{asymptotically optimal} if $\Maxload^f_{\rho n}([n])\le(1+o(1))\Maxload^g_{\rho n}([n]),$ for any strategy $g$, with high probability.

Here we describe and analyse an optimal two-thinning strategy which we call the \emph{$\ell$-threshold} strategy. This is the two-thinning strategy which rejects a ball whenever the number of primary allocations to the suggested bin is at least $\ell$. Our main result is the following,
\begin{thm}\label{thm:main}
Let $f$ be the $\sqrt{\frac{2\log n}{\log\log n}}$-threshold strategy for the allocation of $\lfloor\rho n\rfloor$ balls into $n$ bins. Then $f$ is asymptotically optimal and, with high probability,
\begin{equation*}
\Maxload^f_{\rho n}([n]) = (1+o(1))\sqrt
{\frac{8\log n}{\log\log n}}.
\end{equation*}
\end{thm}

\subsection{Discussion}

\textbf{Balls-and-bins, two-choices and two-thinning.}
It is well known that if each of $\lfloor\rho n\rfloor$ balls is allocated independently to a uniformly chosen random bin in $[n]=\{1,\dots, n\}$, then the most heavily loaded bin contains $\frac{\log n}{\log \log n}+O(1)$ balls with high probability. In their seminal paper, Azar, Broder, Karlin and Upfal \cite{ABKU} have shown that a significantly lower maximum load of $\log_2 \log n+O(1)$ balls could be achieved, with high probability, in a \emph{two-choices} setting, i.e., if the allocation of each ball is governed by an overseer who is offered a choice between two independent, uniformly chosen random bins. Moreover, the overseer can achieve this simply by following a na\"{\i}ve strategy of always selecting the less loaded of the two bins.

The two-thinning setting, considered in this paper, is intermediate between two-choices and no-choice, as it is equivalent to a two-choices setting in which the overseer is oblivious of the location of one of the two available bins. The name ``two-thinning'' is due to yet another point of view on this setting. According to this view an infinite sequence of allocations has been drawn independently and uniformly at random, and the overseer is allowed to thin it on-line (i.e., delete some of the allocations depending only on the past), as long as at most one of every two consecutive entries is deleted (for a more thorough discussion of the model see joint work with Ramdas and Dwivedi \cite{DFGR}, where the model was introduced).

From Theorem~\ref{thm:main} we see that the optimal maximum load under two-thinning is indeed intermediate between the maximum load without thinning and the maximum load in the two-choices setting.

\textbf{More choice.}
Already in \cite{ABKU}, Azar et al. showed that allowing the overseer choice between $k>2$ choices, reduces the asymptotic maximal load by merely a factor of $\log(k)$.
Nonetheless, we make the following conjecture.
 \begin{conjecture}
 In the two-thinning setting, allowing the overseer to iteratively reject up to $k$ suggested allocations for each ball will result in an improved asymptotically optimal maximum load of $$\Theta\left(\left(\frac{\log n}{\log\log n}\right)^{1/(k+1)}\right).$$
 \end{conjecture}

\textbf{More balls.}
Berenbrink, Czumaj, Steger and V\"ocking \cite{BCSV} have considered the power of two choices in the heavily loaded case of the balls and bins model,
that is, when $\omega(n)$ balls are allocated into $n$ bins. They showed that in this case under the power of $k$-choices, the deviation of the maximum load from the average load is asymptotically almost surely
${\log_k\log n}+O(1)$ (see Talwar and Wieder \cite{TW}, for a simpler proof).
We conjecture that the same phenomenon will occur for two-thinning. Namely,
 \begin{conjecture}
 In the two-thinning setting, where $m=\Omega(n)$ balls are two-thinned, the asymptotically optimal maximum load
is $\frac{m}{n}+\Theta\left(\sqrt{\frac{\log n}{\log\log n}}\right)$.
 \end{conjecture}

\textbf{1+$\beta$-thinning.} In his thesis \cite{Mitzenmacher}, Mitzenmacher suggested considering a variant of the power of two-choices in which,
for each allocation independently, there is some small probability that a decision opposite to that made by the overseer will be executed. This notion was recently formulated and studied by Peres, Talwar and Wieder \cite{PTW}, viewing it as having two-choices with probability $\beta$ and no-choice with probability $(1-\beta)$, independently for every ball.
Once errors of this nature are introduced to the model, two-choices and one-retry are equivalent up to a parameter change, and in lightly loaded case of $\lfloor\rho n\rfloor$
 balls allocated into $n$ bins, both offer no improvement over having no-choice at all (see \cite{DFGR} for more details).

\section{Preliminaries} \label{sec:preliminaries}

We take advantage of a comparison lemma of Mitzenmacher and Upfal \cite[Corollary 5.11]{MU}, which we reproduce here, relating the balls-and-bins model with independent Poisson random variables.
Denote by $\N_0$ the set of natural numbers together with $0$. Given two vectors $x,y\in (\N_0)^n$ we write $x\le y$ if $x_i \le y_i$ for all $i\in [n]$. A set $S\subset (\N_0)^n$ is called \emph{monotone decreasing (increasing)} if $x\in S$ implies $y\in S$ for all $y\le x\ (y\ge x)$.

\begin{lemma}[Mitzenmacher and Upfal]\label{Lem:mu}
Let $(X_m)_{m\in[n]}$ be the number of balls in the $m$-th bin when $t$ balls are independently and uniformly allocated into $n$ bins. Further let $(Y_m)_{m\in[n]}$ be independent Poisson$(\frac{t}{n})$ random variables, and let $S$ be a monotone set (either increasing or decreasing). Then
$$\Pr\Big((X_1,\dots,X_n)\in S \Big) \le 2\Pr\Big((Y_1,\dots,Y_n)\in S\Big).$$
\end{lemma}

We also utilise two corollaries of this lemma.

\begin{lemma}\label{lem:retries lead to high bin}
Let $(X_m)_{m\in[n]}$ be the number of balls in the $m$-th bin when $(\theta n)$-balls are independently and uniformly allocated into $n$-bins, for $\theta\in[0,1]$. Then, for any $a \in [\theta n]$ and $S\subset[n]$ we have
$$\Pr\left(\max_{m\in S}(X_m)<a \right) \le 2\exp\left(-\frac{\theta^a|S|}{e a!}\right)$$
\end{lemma}
\begin{proof}
Let $\{Y_m\}_{m\in [n]}$ be i.i.d. Poisson($\theta$) random variables. By Lemma~\ref{Lem:mu} we have
\begin{align*}
\Pr\left(\max_{m\in S}(X_m)<a\right) &\le 2\, \Pr\left(\max_{m \in S} (Y_m) < a\right) = 2 \Pr\left(Y_1 < a\right)^{|S|} \le 2\left(1-e^{-\theta} \frac{\theta^a}{a!}\right)^{|S|} \\
&\le 2\left(1- \frac{\theta^a}{ea!}\right)^{|S|}\le 2\exp\left(-\frac{ \theta^a|S| }{ea!}\right)\ .
\end{align*}
\end{proof}

\begin{lemma}\label{lem:saturation without retry}
Let $(X_m)_{m\in[n]}$ be the number of balls in the $m$-th bin when $(\theta n)$-balls are independently and uniformly allocated into $n$-bins, for  $\theta\in[0,1]$. Then, for any $S\subset[n]$ we have
$$\Pr\left(|\{m\in S : X_m>0\}|\le \frac{\theta|S|}{2e }\right)\le2\exp\left(-\frac{\theta^{2}|S|}{2e^2}\right)\ .$$
\end{lemma}

\begin{proof}
Let $\{Y_m\}_{m\in[n]}$ be i.i.d. Poisson($\theta$) random variables. By Lemma~\ref{Lem:mu}
$$\Pr\left( |\{m\in S : X_m>0\}| \le \frac{\theta|S|}{2e} \right)\le 2 \Pr\left(|\{m\in S : Y_m>0|\le \frac{\theta|S|}{2e}\right).$$
We observe that $\Pr(Y_1 >0)\ge \frac{\theta}{e}$. Using Hoeffding bound for the tail of binomial distributions (see, e.g. \cite[Proposition 1.12]{Dudley}), we obtain,
\begin{align*}
\Pr\left(|\{m\in S : Y_m>0\}|\le \frac{|S|\theta}{2e}\right)\le
\exp\left(-2|S|\left(\frac{\theta}{e}-\frac{\theta}{2e}\right)^2\right)
=\exp\left(-\frac{\theta^{2}|S|}{2e^2}\right).
\end{align*}

\end{proof}

%
%
%
\section{Notation}
Given a thinning strategy $f$, generate $\{Z_t\}_{t\in\N_0}$, the sequence of allocations, in the following way. Let $\{Z^0_t\}_{t\in\N_0}$ and $\{Z^1_t\}_{t\in\N_0}$ be two sequences of independent random variables uniformly distributed in $[n]$. Here $Z^0_t$ represents the primary allocation of the $t$-th ball, while $\{Z^1_t\}_{t\in\N_0}$ is used as a pool of secondary allocations.
Denote by $r_t$ the number of rejections among the first $(t-1)$ primary allocations.
For the $t$-th allocation, look at the history of the process up to time $(t-1)$ and at $Z^0_t$ and apply $f$ to determine whether to accept or reject the primary allocation. If the primary allocation is accepted then set $Z_t$ to be $Z^0_t$ while if it is rejected, then set $Z_t$ to be $Z^1_{r_t}$.

We introduce the following notation. For any $t\le \rho n$, $m\in [n]$ denote
\begin{align*}
F_t(m)&=|\{1\le i \le t\ :\ Z_i=m\}|,\\
A_t(m)&=|\{1\le i \le t\ :\ Z^0_i=m\}|, \\
B_t(m)&=|\{1\le i \le t\ :\ Z^1_i=m\}|.
\end{align*}
In addition, for $\ell\in \N_0$ and for $S\subset [n]$ denote
\begin{align*}
\phi_{t}^\ell(S)&=|\{m\in S\ :\ F_t(m)\ge \ell\}|\\
\alpha_{t}^\ell(S)&=|\{m\in S\ :\ A_t(m)\ge \ell\}|\\
\beta_{t}^\ell(S)&=|\{m\in S\ :\ B_t(m)\ge \ell\}|,
\end{align*}
setting $\phi_{t}^\ell=\phi_{t}^\ell([n])$,
$\alpha_{t}^\ell=\alpha_{t}^\ell([n])$,
$\beta_{t}^\ell=\beta_{t}^\ell([n])$.
Finally, denote
\begin{align*}
\Maxload^f_t(S)&=\max_{m\in S} F_t(m)\\
\Maxload^f_t&=\Maxload^f_t([n]).
\end{align*}

\section{Upper bound on \texorpdfstring{$\Maxload_{\rho n}^f([n])$}{Maxloadfn}}
For $n\ge 3$, denote $L=\left\lceil\sqrt{2\log n /\log\log n}\right\rceil$. Let $f$ be the $L$-threshold strategy, i.e., the one for which $f_i=1$ if and only if $A_i(Z^0_i)\ge L$. The main statement of this section is the following.
\begin{propos}\label{props : upper bound on maxload}
For any $n\ge n(\rho)$ sufficiently large and any $\eta>0$ the strategy $f$ satisfies
\[\Prob\big(\Maxload_{\rho n} > (2+\eta)L \big)\le 2n^{-\frac{\eta}{4}+\frac{2\log \log \log n}{\log\log n}}+2e^{-\sqrt n}.\]
\end{propos}

Let us begin by reducing the upper bound in theorem~\ref{thm:main} to this proposition.
\begin{proof}[Proof of the upper bound in Theorem~\ref{thm:main}]
We apply Proposition~\ref{props : upper bound on maxload} with $\eta = \frac{9\log\log\log n}{\log\log n}$. Observe that $\eta=o(1)$ and by the proposition we have
$$\Prob\big(\Maxload_{\rho n} > (2+\eta)L \big) \le \exp\left(-\frac{\log n}{4\log\log n}\right)+2e^{-\sqrt n}=o(1).$$
\end{proof}

\begin{proof}[Proof of Proposition~\ref{props : upper bound on maxload}]
Denote $r=r_{\lfloor \rho n\rfloor}$.
Our strategy $f$  guarantees that $A_{\rho n}(m)\le L$ for all $m\in [n]$. Hence, under this strategy
\begin{equation} \label{boundLB0}
\Prob\big(\Maxload_{\rho n} \ge (2+\eta)L\big)\le \Prob\left(\max_{m\in[n]} B_{r}(m) \ge L+\eta L\right) = \Prob\big(\beta_{r}^{L+\eta L}>0\big) \ .
\end{equation}

Let $L'=\lceil L+\eta L \rceil$. Notice that if $\beta^{L'}_{r}>0$ then for any $0\le k\le \rho n$, either $r>k$ or $\beta_{k}^{L'}>0$. Hence, for any $0\le k\le \rho n$ we get
\begin{equation}\label{eq: boundLB}
\Prob\big(\Maxload_{\rho n} \ge (2+\eta)L \big)\le\Pr\Big(r>k \Big)+\Pr\Big(\beta_{k}^{L'}>0 \Big) \ .
\end{equation}

We now bound the two probabilities on the right hand side.

To bound $\Pr(r>k)$, let $\{Y_m\}_{m\in [n]}$ be i.i.d. Poisson($\rho$) random variables and write $$Y:=\sum_{m\in [n]}\max(Y_m-L,0).$$
By Lemma~\ref{Lem:mu} we have
\begin{equation}\label{eq: r n first bound}
\Pr(r>k)\le2\Pr\Big(Y>k\Big).
\end{equation}
For a single Poisson($\rho$) random variable we have for $n\ge 100$,
$$\E\Big(e^{\max(Y_1-L,0)}\Big)\le 1+e^{-\rho} \sum_{\ell=1}^{\infty} \frac{\rho^\ell e^{\ell}}{(L+\ell)!}\le 1+\frac{1}{L!}<\exp\left(\frac{1}{L!}\right) \ .$$
Hence, by Markov's inequality, for $k\ge \frac{2n}{L!}$ we have
\begin{equation}\label{eq: r n second bound}
\Pr\Big(Y>k\Big)=\Pr\Big(e^Y > e^k \Big)\le\exp\left(\frac{n}{L!} - k\right)<\exp\left(-\frac{n}{L!}\right)<e^{-\sqrt n}.
\end{equation}
putting together \eqref{eq: r n first bound} and \eqref{eq: r n second bound} we obtain
\begin{equation}\label{eq:Fk bound}
\Pr(r>k)<2e^{-\sqrt n}.
\end{equation}

Next we bound $\Pr\Big(\beta_{k}^{L'}>0 \Big)$.
Let $\{Y_m\}_{m\in [n]}$ be i.i.d. Poisson$(k/n)$ random variables.
By Lemma~\ref{Lem:mu} we have,
$$\Pr\Big(\beta_{k}^{L'}>0 \Big)\le 2\Pr\Big(\max_{m\in [n]}(Y_m)>L'\Big).$$
For $k\le \frac{3n}{L!}$ and $n\ge 100$ we have
 $$\Pr\Big(Y_1>L'\Big)=e^{-k/n} \sum_{\ell=L'+1}^\infty \frac{(k/n)^{\ell}}{\ell!} \le
 \sum_{\ell=L'+1}^\infty \left(\frac{3}{L!}\right)^{\ell}\le \left(\frac{3}{L!}\right)^{L'}.$$
Taking a union bound, we obtain
\begin{equation}
\Pr\Big(\beta_{k}^{L'}(L')>0\Big)\le 2 n \left(\frac{3}{L!}\right)^{L'} .
\end{equation}
By Stirling's approximation for all $\ell>1$ we have
$\ell!> 3\left(\frac{\ell}e\right)^\ell$. Hence,
\begin{align}
\Pr\Big(\beta_{k}^{L'}>0\Big)&\le 2 n \left(\frac{3}{L!}\right)^{L'}\le 2n\left(\frac{L}{e}\right)^{-LL'} \le 2n\left(\frac{L}{e}\right)^{-(1+\eta)L^2}\notag\\
 &\le 2\exp\Big(\log n -(1+\eta)L^2\big(\log L-1\big)\Big)\notag\\
 &\le2\exp\left(\log n-(1+\eta)\frac{2\log n}{\log\log n}\left(\frac12 \log\log n-\frac12\log\log\log n-1\right)\right)\notag\\
 &\le2\exp\left(-\eta\log n+(1+\eta)\frac{2\log n}{\log\log n}\left(\frac12\log\log\log n+1\right)\right)\notag\\
 &\le2\exp\left(-\eta\log n+(1+\eta)\frac{2\log n \log \log \log n}{\log\log n}\right)\notag\\
 &\le 2n^{-\frac{\eta}{4}+\frac{2\log \log \log n}{\log\log n}},
 \label{eq:Ek bound}
\end{align}
for any $n\ge 100$. Putting \eqref{eq:Fk bound} and \eqref{eq:Ek bound}
into \eqref{eq: boundLB}, the proposition follows.
\end{proof}

\section{Lower bound on \texorpdfstring{$\Maxload_{\rho n}^g([n])$ for any strategy $g$}{Maxloadfn}}

Let $\ell=\ell(n)=\left\lfloor\sqrt{2\log n /\log\log n}\right\rfloor$. In this section we prove the following proposition, from which the lower bound in Theorem~\ref{thm:main} is an immediate corollary.

\begin{propos}\label{props : lower bound on maxload}
Let $\eps,\rho>0$ and $n$ sufficiently large (depending on $\rho$ and $\eps$). For any strategy  $g$ we have
\[\Prob\big(\Maxload^g_{\rho n} < (2-\eps) \ell \big)\le \exp\left(-n^{\eps/5}\right) .\]
\end{propos}

To prove Proposition~\ref{props : lower bound on maxload} we use the following lemma.
\begin{lem}\label{lem : one step}
Let $\eps,\rho>0$ and $n$ sufficiently large (depending on $\rho$ and $\eps$) and
denote $\zeta=\rho/8 e\ell$. For any $1\le k \le 2\ell$, $t>\rho n/ 2\ell$ and $S\subset [n]$ such that $|S|\ge n \zeta^k$ and any strategy $g$, we have
\[\Prob\big(E,F\big)\le \exp(- n^{\eps/4}),\]
where
$E=\{\phi^1_{t}(S) <  n\zeta^{k+1}\}$ and
$F=\{\Maxload^g_t(S) < (2-\eps)\ell -k \}$.
\end{lem}
\begin{proof}
Write $T=n\zeta^{k+1}$ and
denote $E'=\{\alpha^1_t(S)<2T\}$
and $F'=\left\{\beta^{(2-\eps)\ell-k}_{T}(S)=0\right\}$.
By applying Lemma~\ref{lem:saturation without retry} with $\theta=\frac{\rho}{2\ell}$ and observing that $2T \le 2\zeta|S|\le \frac{\theta|S|}{2e}$,
we obtain
\[
\Prob\left(E'\right)\le 2\exp\left(-\frac{\theta^2|S|}{2e^2}\right)\le 2\exp\left(-\frac{8n\rho^{k+2}}{(8e\ell)^{k+2}}\right)\le 2\exp\Big(-n^{1+o(1)}\Big)\ ,
\]
Where in the rightmost inequality we used the fact that $k<2\ell$.
By applying Lemma~\ref{lem:retries lead to high bin} with $a=(2-\eps)\ell -k$ and $\theta=\zeta^{k+1}$,
\[
\Prob(F') \le 2\exp\left(-\frac{\zeta^{(k+1)a}|S|}{e a!}\right)\le 2\exp\left(-\frac{\zeta^{(k+1)a} \zeta^k n}{e a^a}\right)
\le 2\exp\left(-\frac{\zeta^{(k+1)(a+1)} n}{e a^a}\right).
\]
Letting $n$ be large enough, and observing that for such $n$ we have $(a+1)(k+1)\le (1-\eps/2)\ell^2$ we obtain
\[
\Prob(F')\le 2\exp\left(-\frac{(\rho/8e\ell)^{(1-\eps/2)\ell^2}n}{e (2\ell)^{2\ell}}\right)
\le 2\exp\Bigg(-\frac{n}{\ell^{(1-\eps/3)\ell^2}}\Bigg)\le 2\exp(-n^{\eps/3}),
\]
where the two rightmost inequalities use the fact that $\ell^{\ell^2}\ge n$, while
$c^{\ell^2}$ and $\ell^\ell$ are both sub-polynomial in $n$ for any $c>0$.

We claim that $E'^c\cap F'^c\subseteq E^c\cup F^c$.
Indeed, we observe that  $\{r_t\le T\}\cap E'^c\subset E^c$,
while $\{r_t> T\}\cap F'^c\subset F^c$. Hence $E\cap F\subset E'\cup F'$.
From our bounds on $\Prob(E')$ and $\Prob(F')$ the proposition follows.
\end{proof}

\begin{proof}[Proof of Proposition~\ref{props : lower bound on maxload}]
Fix $\eps,\rho>0$ and let $g$ be a thinning strategy.
We divide our process into $s=\lceil(2-\eps)\ell\rceil$ stages each consisting of the allocation of $w=\left\lceil\frac{\rho n}{2\ell}\right\rceil$ balls so that the $k$-th stage process consists of $Z_{(k-1)w+1},\dots,Z_{kw}$. These are followed by a final stage in which the remaining balls are allocated.

Denote $S_k=\{m\in [n]\ :\ A_{kw}(m)\ge k\}$. For $\zeta=\rho/8 e\ell$, we define
$E_{k}=\{|S_k| <  n\zeta^{k}\}$ and $F_{k}=\{\Maxload_{kw}^g < (2-\eps)\ell\}$.

By applying Proposition~\ref{lem : one step} to the $k$-th stage process with $S=S_k$ we obtain that
$$\Pr(E_{k+1}\cap F_{k+1}\ |\ E^c_k)\le \exp(- n^{\eps/4}).$$
The see this, observe that
the size of $S_{k+1}$ is at least the number of bins in $S_k$ which were allocated at least one ball in the $k$-th stage process and that $\Maxload^g_{(k+1)w}$
is at least $k$ plus the maximum number of balls that were allocated in the $k$-th stage process to a single bin in $S_k$.

Observe that $F_{k+1}\subseteq F_{k}$ we use the law of total probability to obtain
$$\Pr(E_{k+1}\cap F_{k+1})=
\Pr(E_{k+1}\cap F_{k+1}\cap E_k)+ \Pr(E_{k+1}\cap F_{k+1}\cap E^c_k)
\le \Pr(E_{k}\cap F_{k}) + \Pr(E_{k+1}\cap F_{k+1}\ |\ E^c_k)$$

Since $E_0\cap F_0 =\emptyset$, we may use induction to deduce that for sufficiently large $n$ we have,
$$\Pr\left(E_{s}\cap F_{s}\right)\le
\sum_{k=1}^{s} \Pr(E_{k}\cap F_{k}\ |\ E^c_{k-1})\le
s\exp(- n^{\eps/4})\le \exp(- n^{\eps/5}).$$

Since $\{\Maxload^g_{\rho n} < (2-\eps) \ell\}\subset E_{s}\cap F_{s}$, this concludes the proof.

\end{proof}

\end{document}